\documentclass[12pt]{amsart}
\usepackage{}
\usepackage{amssymb}
\usepackage{amsmath,amssymb,amsbsy,amsfonts,amsthm,latexsym,
                        amsopn,amstext,amsxtra,euscript,amscd,mathrsfs,color,bm,cite}
\usepackage{float}
\usepackage[english]{babel}
\usepackage{mathtools}
\usepackage{todonotes}
\usepackage{url}
\usepackage[colorlinks,linkcolor=blue,anchorcolor=blue,citecolor=blue,backref=page]{hyperref}
\usepackage{cases}
\usepackage{txfonts}
\usepackage{latexsym}
\usepackage{amsthm}
\usepackage{mathrsfs}
\usepackage{amssymb, amsmath}
\usepackage{cite}
\usepackage{color}
\usepackage{verbatim}

\def\re{\text{Re}}
\def\ma{\mathfrak{a}}
\def\mb{\mathfrak{b}}
\def\mc{\mathfrak{c}}
\def\mA{\mathscr{A}}
\def\mB{\mathscr{B}}

\def\mE{\mathscr{E}}

\def\mS{\mathscr{S}}
\def\mJ{\mathcal{J}}

\def\mH{\mathcal{H}}
\def\mM{\mathcal{M}}
\def\mmM{\mathscr{M}}
\def\mE{\mathcal{E}}

\def\mI{\mathcal{I}}
\def\r){\right)}
\def\B{\Bigg}

\def\d{\mathrm{d}}

\def\ssum{\mathop{\sum\nolimits^*}}

\def\ppmod{\!\!\!\!\!\pmod}
\def\ssqrt{\!\sqrt}
\let\ve=\varepsilon

\let\ol=\overline

\let\vp=\varphi

\let\wh=\widehat

\theoremstyle{definition}

\newtheorem{remark}{Remark}[section]

\theoremstyle{plain}
\newtheorem{theorem}{Theorem}
\newtheorem{corollary}[theorem]{Corollary}
\newtheorem{lemma}[theorem]{Lemma}

\numberwithin{equation}{section}
\numberwithin{theorem}{section}
\def\qed{\ifhmode\textqed\fi
   \ifmmode\ifinner\quad\qedsymbol\else\dispqed\fi\fi}
\def\textqed{\unskip\nobreak\penalty50
    \hskip2em\hbox{}\nobreak\hfil\qedsymbol
    \parfillskip=0pt \finalhyphendemerits=0}
\def\dispqed{\rlap{\qquad\qedsymbol}}

\begin{document}
\title[Mixed moments of twisted $L$-functions]{Mixed moments of twisted $L$-functions}

\author[Z. Tang] {Zhenpeng Tang}
\address{ZT: School of Mathematics, Hefei University of Technology, Hefei 230009, P.R. China}
\email{zptang@mail.hfut.edu.cn}

 \author[X. Wu]{Xiaosheng Wu}
\address{XW: School of Mathematics, Hefei University of Technology, Hefei 230009, P.R. China}
\email{xswu@amss.ac.cn}

 %\author[P. Xi]{Ping Xi}
% \address{PX: School of Mathematics and Statistics, Xi'an Jiaotong University, Xi'an 710049,
%P.R. China}
%\email{ping.xi@xjtu.edu.cn}

\begin{abstract}
We establish an asymptotic formula with a power-saving error term for the twisted mixed moment of Dirichlet $L$-functions and automorphic $L$-functions twisted by all primitive characters modulo $q$, valid for all admissible moduli. As a special case, this extends the asymptotic result of Blomer, Fouvry, Kowalski, Michel, and Mili\'cevi\'c to general moduli, achieving an error term as sharp as the best bound recently proved by Khan and Zhang for prime moduli.
\end{abstract}

\thanks{The authors are supported in part by the NSFC Grant 12271135, Anhui Provincial Natural Science Foundation Grant 2508085J005 and the Fundamental Research Funds for the Central Universities Grant JZ2025HGTG0254}
\maketitle

\section{Introduction}

\subsection{Moments of twisted $L$-functions}
The study of moments in families of $L$-functions constitutes a central problem in modern number theory. As emphasized in the elegant work of Young \cite{You11}, these moments are not only pivotal for their wide-ranging applications but also serve as fundamental objects that unveil deep structural properties and inherent symmetries within the family.

Asymptotic formulae for moments of $L$-function with a power-saving error term are particularly crucial, as they play vital roles in amplification techniques, mollification methods, and resonance techniques. The complexity of such moment calculations can be quantified by the ratio $r=\log \mathcal{C}/\log |\mathcal{F}|$, where $|\mathcal{F}|$ denotes the size of the family and $\mathcal{C}$ its analytic conductor. Notably, computational difficulty increases with growing complexity. The threshold $r=4$ is the precise boundary, where most current analytic techniques fall just short of producing an asymptotic formula. In the few successful cases, some deep input is typically indispensable; see \cite{CLMR24} \cite{CIS12} \cite{IS00} \cite{Kha12} \cite{KMV00} \cite{Li24} for example.

Let $f$ and $g$ denote two fixed (holomorphic or non-holomorphic) Hecke eigenforms of level $1$, not necessarily cuspidal. The following second moment is defined in the context $r=4$: for $q\nequiv 2\pmod4$,
\begin{align}\label{eqgmoment}
M_{f,g}(q)=\frac1{\vp^*(q)}\ssum_{\chi\ppmod q}L\left(\tfrac12,f\otimes\chi\r)\ol{L\left(\tfrac12,g\otimes\chi\r)},
\end{align}
where the sum runs over all primitive characters modulo $q$, and $\vp^*(q)$ denotes the number of such characters. Asymptotics for the moment \eqref{eqgmoment} with a power-saving error term are crucial to amplification and related analytic techniques in problems such as subconvexity, nonvanishing, and extreme values of $L$-functions; see \cite{BFKMMS23} for a rich sample of applications.

When both $f$ and $g$ correspond to non-cuspidal Eisenstein series, the moment \eqref{eqgmoment} reduces to
\[
M_{E,E}(q)=\frac1{\vp^*(q)}\ssum_{\chi\ppmod q}\left|L\left(\tfrac12,\chi\r)\right|^4,
\]
the fourth moment of Dirichlet $L$-functions, which has a very long history and has been extensively studied; see \cite{BFK+17a} \cite{HB81} \cite{Sou07} \cite{Wu23} \cite{You11} for example. The first asymptotic formula with a power-savings error term was established by Young \cite{You11} for prime moduli. Wu \cite{Wu23} proved the asymptotic formula for all admissible moduli, and the best-known error term to date is $O\left(q^{-\frac1{20}+\ve}\r)$; see \cite{BFK+17b} \cite{GWZ25}.

When both $f$ and $g$ are cuspidal (holomorphic or Maa\ss) in \eqref{eqgmoment}, some special cases (ie. $f=g$) were studied earlier by Stefanicki \cite{Ste96} and Gao, Khan, and  Ricotta \cite{GKR09}. Blomer and Mili\'cevi\'c \cite{BM15} established a power-saving asymptotic for most moduli, specifically when $q$ is not close to a prime or to a product of two primes of comparable size. The main obstacle in these cases has been the lack of power-saving estimates for bilinear forms involving Kloosterman sums in the P\'olya-Vinogradov range.
The case of prime moduli was later addressed by Kowalski, Michel, and Sawin \cite{KMS17}, who proved the asymptotic with a power-saving error term $O\left(p^{-\frac1{144}+\ve}\r)$. The remaining case was recently resolved by  Mili\'cevi\'c, Qin, and Wu \cite{MQW25} and independently by  Pascadi~\cite{Pascadi2025}. Both works \cite{MQW25} and \cite{Pascadi2025} established a power-saving asymptotic for all admissible moduli $q$, with the former also removing the dependence on the Ramanujan--Petersson conjecture and achieving a sharper error term $O\left(q^{-\frac1{216}+\ve}\r)$.

These are termed mixed moments when one eigenform is cuspidal while the other corresponds to a non-cuspidal Eisenstein series. The first power-saving asymptotic formula for this case was established by Blomer, Fouvry, Kowalski, Michel, and Mili\'cevi\'c \cite{BFK+17a}, who proved that for prime moduli $p$
\[
M_{f,E}(p)=\frac1{\vp^*(p)}\ssum_{\chi\ppmod p}L\left(\tfrac12,f\otimes\chi\r)\ol{L\left(\tfrac12,\chi\r)}^2=\frac{L(1,f)^2}{\zeta(2)}+O\left(p^{-\frac1{68}+\ve}\r).
\]
The error term was subsequently improved to $O\left(p^{-\frac1{64}+\ve}\r)$ by Shparlinski \cite{Shp19}, and more recently to $O\left(p^{-\frac1{22}+\ve}\r)$ for holomorphic $f$ and to $O\left(p^{-\frac5{152}+\ve}\r)$ for Maa\ss\ forms $f$ by Khan and Zhang \cite{KZ23}.

In this paper, we investigate the mixed moment for arbitrary moduli. More precisely, we study the following twisted mixed moment
\[
M_{f,E}(q;a,b)=\frac1{\varphi^*(q)}\ssum_{\chi\ppmod q} L\left(\tfrac12,f\otimes\chi\r)L\left(\tfrac12,\ol\chi\r)^2\chi(\ol{a}b),
\]
for any integers $a,b\le q$ satisfying $(a,b)=(ab,q)=1$.
%To state our results, we define
%\begin{align}
%c_{a,b}=\sum_{a_1\mid a^\infty}\sum_{b_1\mid b^\infty} \frac{\lambda(ba_1b_1)\tau(aa_1b_1)}{a_1b_1} ,\label{eqc1c2}\\
%\end{align}
%
%\[
%\sum_{a_1\mid a^\infty} \frac{\lambda(a_1)\tau(aa_1)}{a_1}
%\]
%
%\begin{align}
%c_1(a)&=\prod_{p\mid a}\left(1-\frac{\lambda(p)}{p}+\frac1{p^2}\r)\left(1-\frac1{p^2}\r)^{-2}\sum_{a^\#}\frac{\lambda(aa^\#)\tau(a^\#)}{a^\#},\label{eqc1c2}\\ c_2(b)&=\prod_{p\mid b}\left(1-\frac{\lambda(p)}{p}+\frac1{p^2}\r)\left(1-\frac1{p^2}\r)^{-2}\sum_{b^\#}\frac{\lambda(b^\#)\tau(bb^\#)}{b^\#},\label{eqc2c1}
%\end{align}
%where the sums run through all positive integers $a^\#, b^\#$ whose prime factors are entirely contained in $a$ and $b$ respectively. {\red These coefficients satisfy the uniform bound $c_1(a), ~c_2(a)\ll a^\ve$.}
\begin{theorem}\label{themain}
Let $f$ be a cuspidal Hecke eigenform of level $1$ with Hecke eigenvalues $\lambda(n)\ll \tau(n)n^{\theta_f}$. If $f$ is Maa\ss~form we further assume  that its root number satisfies $\ve(f)=1$. For $q\nequiv 2 \pmod4$ and $1\le a,b\le q$ satisfying $(a,b)=(ab,q)=1$, we have
\begin{align*}
M_{f,E}(q;a,b)=&c_f\prod_{p\mid qab}\left(1-\frac{\lambda(p)}{p}+\frac1{p^2}\r)\left(1-\frac1{p^2}\r)^{-2}\\
&\times\frac{c_{a,b}+c_{b,a}}{(ab)^{\frac12}}\frac{L(1,f)^2}{\zeta(2)} +O\left(q^{-\frac{1}{20}+\ve}a^{\frac{3}{10}}+q^{-\frac{1-2\theta_f}{22+16\theta_f}+\ve}b^{\frac{3+2\theta_f }{11+8\theta_f}}\r),
\end{align*}
where the coefficients $c_f$, $c_{a,b}$ are defined by
\begin{align}
c_f=
\begin{cases}
1/2,  &\text{if} \ f\ \text{is holomorphic},\\
1, &\text{if}\ f\ \text{is a Maa\ss\ form},
\end{cases}
\quad c_{a,b}=\sum_{a_1\mid a^\infty}\sum_{b_1\mid b^\infty} \frac{\lambda(aa_1b_1)\tau(ba_1b_1)}{a_1b_1}. \label{eqc1c2}
\end{align}
\end{theorem}

\begin{remark}
The dependence on  $\theta_f$ in the error term of Theorem \ref{themain} arises from the trivial estimate \eqref{eqtbB}, which constitutes the sole component requiring an upper bound for the Hecke eigenvalues of $f$. The assumption $\ve(f)=1$ for a Maa\ss~form is natural, as otherwise $M_{f,E}(q;a,b)=0$ for parity reasons (see Remark \ref{rem1}). In the Maa\ss~case, the coefficient $c_f$ addresses a literature discrepancy originating from a typographical oversight in \cite[Theorem 1.2]{BFK+17a} and was subsequently replicated in \cite[Thoerem 1.1]{KZ23}.
\end{remark}

Setting $a=b=1$ in Theorem \ref{themain} results in a power-saving asymptotic for $M_{f,E}(q)$ across all admissible moduli $q$.

\begin{corollary}\label{cor}
Let $f$ be a cuspidal Hecke eigenform of level $1$ with Hecke eigenvalues $\lambda(n)\ll \tau(n)n^{\theta_f}$. If $f$ is Maa\ss~form we further assume that its root number satisfies $\ve(f)=1$. For $q\nequiv 2 \pmod4$, we have
\begin{align*}
M_{f,E}(q)=c_f\prod_{p\mid q}\left(1-\frac{\lambda(p)}{p}+\frac1{p^2}\r)\left(1-\frac1{p^2}\r)^{-2}\frac{L(1,f)^2}{\zeta(2)}+O\left(q^{-\frac{1-2\theta_f }{22+16\theta_f }+\ve}\r),
\end{align*}
where $c_f$ is the same constant as previously defined.
\end{corollary}
For the holomorphic case, Deligne's well-known work \cite{Del74} gives $\theta_f=0$, whereas for the Maa\ss~form we may take $\theta_f= 7/{64}$ due to Kim--Sarnak \cite{Kim03}. This yields an error term as sharp as the best bound \cite{KZ23} known for prime moduli.

%The treatment of these three cases employs distinct methods based on the special properties of $\lambda(n)$ and $\tau(n)$.
%When both $f$ and $g$ are cuspidal, we can exploit wilton's bound for the sum of $\lambda(n)$ twisted by additive characters, which provides square root cancelation. This allows application of Jutila's circle method with an almost free moduli parameter $C$, as demonstrated in the work of Blomer and Mili\'cevi\'c \cite{BM15}.
%For the other two cases, the approach fundamentally relies on the multiplicative decomposition $\tau(n)=\sum_{n_1n_2=n}1$. This identity transforms the original sum involving $\tau(n)$ into two sums with smooth coefficients, making them amenable to Fourier--analysis techniques (see also \cite{BFK+17b}).

Through the functional equations of $L$-functions and the orthogonality of characters, the evaluation of the moment $M_{f,E}(q,a,b)$ reduces to analyzing the shifted convolution sum
\begin{align}\label{eq+1}
\sum_{\substack{bm\equiv \pm an\ppmod d\\(mn,q)=1}}\lambda(m) \tau(n)W\left(\frac mM\right)W\left(\frac {n}{N}\right),
\end{align}
where $a, b,M,N\ge1$ and $W$ is smooth test functions. The classical approach to such shifted convolution sum involves either the $\delta$-method or Jutila's circle method to eliminate the congruence condition (see \cite{BM15}). For \eqref{eq+1}, a more efficient approach is to expand the divisor function using the multiplicative decomposition $\tau(n)=\sum_{n_1n_2=n}1$, which decomposes the $n$-sum into two sums with smooth coefficients, making them amenable to Fourier-analytic techniques.

The diagonal terms (where $bm=an$) in \eqref{eq+1} contribute to the main term, which can be evaluated explicitly using standard techniques. The primary challenge arise from the off-diagonal terms, which we classify into two distinct cases based on the relative magnitudes of $bM$ and $aN$: the balanced case and the unbalanced case. These two cases demand fundamentally different approaches.

In the balanced case where $bM$ and $aN$ are of comparable size, we first remove the coprimality condition and reduce the problem to estimating the sum
\begin{align}\label{eqAdef}
A_q(a,b,M,N)=\sum_{\substack{bm\equiv \pm an \ppmod q\\ bm\neq an}}\lambda(m) \tau(n) W\left(\frac {bm}M\r) W\left(\frac {an}N\r),
\end{align}
where $W$ are some smooth functions.
For $M\ge N$, we establish the bound
\[
A_q(a,b,M,N)\ll  q^\ve\B(\frac{M}{q^{\frac12}}+\frac{(ab,q)^{\frac14}M^{\frac54}N^{\frac14}}{(ab)^{\frac14}q} +\frac{M^{\frac34}N^{\frac14}}{(ab)^{\frac14}q^{\frac14}} +\frac{(ab,q)^{\frac14}MN^{\frac12}}{(ab)^{\frac12}q^{\frac34}}\B)
\]
with an analogous result holding for $N\ge M$. Compared to the bound in \cite[(8.13)]{BM15}, our result achieves two key improvements: reducing the influence of the factor $(ab,q)^{\frac12}$ and recovering a saving $(ab)^{-\frac14}$ in the third term. This improvements stem from two innovation: the elimination of dependence on the Ramanujan--Petersson conjecture and the minimization of the impact of $a, b$ through new upper bounds for sums involving Kloosterman sums, as established in the recent work of Gao, Wu, and Zhao \cite{GWZ25}.

% new ingredient to avoid the
%
%our new estimate (see Theorem \ref{thmsls})
%\begin{align*}
%\sum_{\substack{f\in\mathcal{B}(Q)\\ |\kappa_f|\le K}}\frac1{\cosh(\pi\kappa_f)}\B|\sum_{n}a_n\ssqrt{sn}~\rho_f(sn)\B|^2\ll (sQKN)^\ve\B(K+\frac{s^{\frac12}}{Q^{\frac12}}\B)\B(K+\frac{(s,Q)^{\frac12}N}{Q^{\frac12}}\B)N\|a_n\|_\infty
%\end{align*}
%to avoid the Ramanujan-Peterson conjecture. Compared to \cite[Theorem 13]{BM15}, our estimate eliminates the requirement $(n,sQ)=1$ in the summation and nearly reduces the dependence on $(s,Q)$, producing a cleaner and widely applicable result that covers nearly all situations. Based on this, we extend Deshouillers and Iwaniec's class result \cite[Theorem 8]{DI82} on Kloosterman sums to a more general form, presented as Theorem \ref{thKls} below.

For the unbalanced case where $BM$ and $AN$ differ significantly in size, the theory of algebraic trace functions becomes inapplicable due to its restriction to prime moduli. To address this, we employ the multiplicative decomposition of the divisor function $\tau(n)$, which allows us to reformulate the problem as analyzing a bilinear sum involving incomplete Kloosterman sum. The estimate of this bilinear form relies primarily on a recent result by Kerr, Shparlinski, Wu, and Xi \cite{KSWX23} concerning arbitrary moduli.
The approach also requires the Weil bound for Kloosterman sums and several Fourier-analytic tools, including the Poisson summation formula and the Voronoi summation formula.

\subsection{Notation}
We adopt the standard convention that $\ve$ denotes an arbitrarily small positive constant whose values may vary from one occurrence to another. All implied constants in our estimates may depend on $\ve$ as well as on the fixed cusp form $f$.
Throughout this paper, $W(x)$ denotes a smooth complex-valued function with compact support in $[1/2,3]$, satisfying
\begin{align}\label{eqW}
W^{(j)}(x)\ll_{j}1,
\end{align}
for all integers $j\ge0$. The function $W(x)$ may also vary at each occurrence. We call a modulus $q$ admissible if $q\nequiv 2\pmod4$, which is a natural assumption since no primitive characters exist modulo $q$ otherwise.

\section{Background and auxiliary lemmas}

\subsection{Dirichlet $L$-functions}
Let $\chi$ be a primitive Dirichlet character modulo $q$. For $\re(s)>1$, its associated Dirichlet $L$-function is defined by
\[
L(s,\chi):=\sum_{n}\frac{\chi(n)}{n^{s}}=\prod_{p}\left(1-\frac{\chi(p)}{p^s}\r)^{-1}.
\]
We denote the completed $L$-function by
\[
\Lambda(s,\chi)=q^{-\frac s2}L_\infty(s,\chi)L(s,\chi),
\]
where
\begin{align*}
L_{\infty}(s,\chi)=\pi^{-\frac s2}\Gamma\left(\frac{s+\ma}2\r),\ \ \ \ \ma=\begin{cases}
0,  &\text{if}\ \chi(-1)=1,\\
1, &\text{if}\ \chi(-1)=-1.
\end{cases}
\end{align*}
Through analytic continuation to $\mathbb{C}$, this completed $L$-function satisfies the functional equation:
\[
\Lambda(s,\chi)=\ve(\chi)\Lambda(1-s,\ol{\chi})
\]
with
\[
\ve(\chi)=i^{-\ma}\ve_\chi, \ \ \ \ \ \ve_\chi=\frac{1}{\ssqrt{q}}\sum_{x\ppmod q}\chi(x)e\left(\frac xq\r).
\]

\subsection{Automorphic preliminaries}
We outline some foundamental facts about automorphic forms, mainly following \cite{BM15}, \cite{DI82}, and \cite{GWZ25}.
Let $\mb$ be a cusp of $\Gamma_0(Q)$ with scaling matrix $\sigma_{\mb}$. For a holomorphic modular form $f$ of level $Q$ and weight $k$, its Fourier expansion around $\mb$ takes the form
\[
f(\sigma_{\mb}z)=\sum_{n\ge1}\rho_f(\mb,n)(4\pi n)^{\frac k2}e(nz).
\]
For a Maa\ss\  form $f$ with spectral parameter $t$, the expansion becomes
\[
f(\sigma_{\mb}z)=\sum_{n\neq0}\rho_f(\mb,n)W_{0,it}(4\pi|n|y)e(nx),
\]
where $W_{0,it}(y)=(y/\pi)^{\frac12}K_{it}(y/2)$ is a Whittaker function.
Each cusp $\mc$ of $\Gamma_0(Q)$ has an Eisenstein series $E_{\mc}(\sigma_\mb z,s)$ with Fourier expansion at $s=\frac12+it$:
\[
E_{\mc}(\sigma_\mb z, \tfrac12+it)=\delta_{\mb,\mc}y^{\frac12+it}+\varphi_{\mb,\mc}(\tfrac12+it)y^{\frac12-it} +\sum_{n\neq0}\rho_{\mb,\mc}(n,t)W_{0,it}(4\pi|n|y)e(nx).
\]
When $\mb=\infty$, we simplify notation by writing $\rho_f(n)$ and $\rho_{\mc}(n,t)$ for the Fourier coefficients.

For a cuspidal newform $f$ with normalized Hecke eigenvalues $\lambda(n)$, the scaling relation
\[
\lambda(n)\rho_f(1)=\sqrt{n}\rho_f(n)
\]
and the multiplicative formula
\begin{align}\label{eqlambdam}
\lambda(mn)=\sum_{d\mid(m,n)}\mu(d)\chi_0(d)\lambda\left(\frac md\r)\lambda\left(\frac nd\r), \quad m, n\in \mathbb{N}
\end{align}
hold, where $\chi_0$ is the trivial character modulo $Q$.
The Ramanujan--Petersson conjecture states that
\[
|\lambda(n)|\le \tau(n).
\]
Deligne \cite{Del74} proved this for holomorphic newforms, whereas for Maa\ss\  newform, Kim--Sarnak \cite{Kim03} established the weaker bound
\[
|\lambda(n)|\le \tau(n)n^{\theta} \ \ \ \text{with}\ \ \ \theta =\tfrac 7{64}.
\]
The Ramanujan-Petersson conjecture holds on average:
\begin{align}\label{eqal}
\sum_{n\le x}|\lambda(n)|^2\ll_f x.
\end{align}

For non-newforms, the exact multiplicativity of their Fourier coefficients in \eqref{eqlambdam} no longer holds, but effective averaged substitutes are available, particularly for specially constructed bases. As in \cite[Section 11]{GWZ25}, we denote by $\mathcal{B}_k(Q)$ the special bases for the space of holomorphic cusp forms of level $Q$ and weight $k$, and by $\mathcal{B}(Q)$ the bases for the space of Maa\ss~cusp forms of level $Q$. The explicit construction of these bases is detailed in a work by Blomer, Harcos, and Michel \cite[Section 3.1]{BHM07}.
%Throughout this paper, we fix $\mathcal{B}_k(Q)$ and $\mathcal{B}(Q)$ to be these specially constructed bases; see also \cite[Section 11]{GWZ25}.

\subsection{Twisted $L$-functions}
We review some basic facts on twisted $L$-functions in this section; for further details, see \cite[Section 2.3]{BFK+17a}.
Let $\chi$ be a primitive Dirichlet character modulo $q$ and let $f$ be a cuspidal Hecke eigenform with Hecke eigenvalues $\lambda(n)$. The twisted form $f\otimes\chi$ remains cuspidal for the congruence subgroup $\Gamma_0(q^2)$ with nebentypus $\chi^2$ (see e.g \cite[Propositions 14.19 \& 14.20]{IK04}). Its $L$-function is given by
\[
L(s,f\otimes\chi)=\sum_{n\ge1}\frac{\lambda(n)\chi(n)}{n^{s}}=\prod_p\left(1-\frac{\lambda(p)\chi(p)}{p^s}+\frac{\chi^2(p)}{p^{2s}}\r)^{-1},\ \ \text{for}\ \ \re(s)>1,
\]
where the product is over primes $p$.
The completed $L$-function is defined by
\begin{align*}
\Lambda(s,f\otimes\chi)=q^s L_\infty(s,f\otimes\chi)L(s,f\otimes\chi),
\end{align*}
with
\begin{align*}
L_\infty(s,f\otimes\chi)=\begin{cases}
(2\pi)^{-\frac{k-1}2-s}\Gamma\left(\frac{k-1}2+s\r),  &\text{if} \ f\ \text{is holomorphic of weight}\ k,\\
\pi^{-s-\ma}\Gamma\left(\frac{s+i\kappa+\ma}2\r)\Gamma\left(\frac{s-i\kappa+\ma}2\r), &\text{if}\ f\ \text{is a Maa\ss\ form with eigenvalue}\ \tfrac14+\kappa^2.
\end{cases}
\end{align*}
This completed $L$-function admits analytic continuation to $\mathbb{C}$ and satisfies the functional equation (see e.g \cite[Theorem 14.17, Proposition 14.20]{IK04}):
\[
\Lambda(s,f\otimes\chi)=\ve(f\otimes\chi)\Lambda(1-s,f\otimes\ol{\chi}),
\]
where the root number is given by
\begin{align*}
\ve(f\otimes\chi)=\begin{cases}
\ve(f)\ve_\chi^2,  &\text{if} \ f\ \text{is holomorphic},\\
\chi(-1)\ve(f)\ve_\chi^2, &\text{if}\ f\ \text{is a Maa\ss\ form},
\end{cases}
\end{align*}
and $\ve(f)=\pm 1$ denotes the root number of $L(s,f)$.

\subsection{Approximate functional equation}
Combining the functional equations of both $L(s,\chi)$ and $L(s,f\otimes\chi)$ results in the following functional equation:
\begin{align}\label{eqfe}
\Lambda(s,f\otimes\chi)\Lambda(s,\ol{\chi})^2=\ve(f,\chi)\Lambda(1-s,f\otimes\ol\chi)\Lambda(1-s,\chi)^2,
\end{align}
where
\begin{align*}
\ve(f,\chi)=\begin{cases}
\chi(-1)\ve(f),  &\text{if} \ f\ \text{is holomorphic},\\
\ve(f), &\text{if}\ f\ \text{is a Maa\ss\ form}.
\end{cases}
\end{align*}
\begin{remark}\label{rem1}
Observe that $\ve(f,\chi)$ is independent of $\chi$ when $f$ is a Maa\ss~form.
When $\ve(f,\chi)=\ve(f)=-1$, summing both sides of equation \eqref{eqfe} over all primitive characters leads to the vanishing of $M_{f,E}(q;a,b)$ by symmetry. For holomorphic $f$, the root number $\ve(f,\chi)$ depends on $\chi$ at most though its parity $\chi(-1)$. The symmetry also implies that the total contribution of such characters $\chi$ satisfying $\ve(f,\chi)=-1$ to $M_{f,E}(q;a,b)$ is zero. Therefore, we restrict our consideration to the following cases:
\begin{itemize}
  \item for a Maa\ss~form $f$, we take $\ve(f)=1$,
  \item for a holomorphic $f$, we consider only characters with $\chi(-1)=\ve(f)$.
\end{itemize}
In both cases, it holds that
\[
\ve(f,\chi)=1,
\]
which we assume henceforth.
\end{remark}

Given that $\ve(f,\chi)=1$, standard techniques \cite[Theorem 5.3]{IK04} yield
the following approximate functional equation (see also \cite[(2.6)]{BFK+17a}):
\begin{align}\label{eqafe}
L\left(\tfrac12,f\otimes\chi\r)L\left(\tfrac12,\ol\chi\r)^2=\sum_{m,n\ge1}&\frac{\lambda(m) \tau(n)}{(mn)^{\frac12}}\chi(m)\ol{\chi}(n)V_{f,\ma}\left(\frac{mn}{q^2}\r)\\
&+\sum_{m,n\ge1}\frac{\lambda(n) \tau(m)}{(mn)^{\frac12}}\chi(m)\ol{\chi}(n)V_{f,\ma}\left(\frac{mn}{q^2}\r),\notag
\end{align}
where the weight function takes the form
\begin{align}\label{eqV}
V_{f,\ma}(x)=\frac1{2\pi i}\int_{(2)}\frac{L_\infty\left(\frac12+s,f\otimes\chi\r)L_\infty\left(\frac12+s,\ol{\chi}\r)^2} {L_\infty\left(\frac12,f\otimes\chi\r)L_\infty\left(\frac12,\ol{\chi}\r)^2} x^{-s}\frac{\d s}s,
\end{align}
which depends on $\chi(-1)$ only through its parity and is otherwise independent of $\chi$. The function $V_{f,\ma}(x)$ is smooth, decays rapidly for $x\gg q^\ve$ and is approximately $1$ for small $x$.

\subsection{The orthogonality formula}
The orthogonality relation for averages over primitive Dirichlet characters is given by the following lemma:
\begin{lemma}
If $(mn,q)=1$, then for $\sigma\in \{-1,1\}$
\begin{align}\label{eqorth}
\ssum_{\substack{\chi\ppmod q\\ \chi(-1)=\sigma}}\chi(m)\ol{\chi}(n)=\frac12\B(\sum_{d\mid(q,m-n)}\vp(d)\mu\left(\frac qd\r)+\sigma \sum_{d\mid(q,m+n)}\vp(d)\mu\left(\frac qd\r)\B).
\end{align}
\end{lemma}
\begin{proof}
This identity follows from standard techniques. For detailed proofs, we refer to \cite{HB81} and \cite{Sou07}.
\end{proof}
\subsection{Averages of Hecke eigenvalues with coprimality}
Let $\lambda(n)$ denote the Hecke eigenvalues of a cuspidal Hecke eigenform $f$ (holomorphic or Maa\ss\ ) and define the arithmetic functions
\[
\varpi_{\lambda}(\delta,q)=\sum_{\substack{kl^2=\delta\\ kl\mid q}}\mu(l)\mu(kl)\lambda(k), \ \ \ \ \ \varpi_{\tau}(\delta,q)=\sum_{\substack{kl^2=\delta\\ kl\mid q}}\mu(l)\mu(kl)\tau(k).
\]
These functions satisfy the uniform bounds
\begin{align}\label{eqvp}
\varpi_{\lambda}(\delta,q)\ll q^\ve(\delta,q)^{\theta_f }, \ \ \ \ \ \ \varpi_{\tau}(\delta,q)\ll q^\ve.
\end{align}

\begin{lemma}\label{lemtl}
For any arithmetic function $F(n)$,
we have
\[
\sum_{(n,q)=1}\lambda(n)F(n)=\sum_{\delta}\varpi_{\lambda}(\delta,q)\sum_n\lambda(n)F(\delta n),
\]
where the sum over $\delta$ contains $\ll q^\ve$ terms.
The same identity holds when replacing $\lambda(n)$ with the divisor function $\tau(n)$.
\end{lemma}
\begin{proof}
Applying M\"{o}bius inversion and the multiplicativity relations \eqref{eqlambdam}, we obtain
\begin{align*}
\sum_{(n,q)=1}\lambda(n)F(n)&=\sum_{d\mid q}\mu(d)\sum_n\sum_{l\mid (d,n)}\mu(l)\lambda\left(\frac dl\r)\lambda\left(\frac nl\r)F(dn)\\
&=\sum_{\delta}\varpi_{\lambda}(\delta,q)\sum_n\lambda(n)F(\delta n)
\end{align*}
by setting $\delta=dl$. The identical argument applies to $\tau(n)$.
\end{proof}

We now present a refined Voronoi formula incorporating a coprimality constraint.
\begin{lemma}[Voronoi formula]\label{lemcV}
Let $b\in \mathbb{Z}$ and $d,q\in\mathbb{N}$ with $(b,d)=1$, and let $V$ be a smooth function compactly supported in $(0,\infty)$.  Then
\begin{align}
\sum_{(n,q)=1}\lambda(n)V(n)e\left(\frac{bn}d\r)=\sum_{\pm}\sum_{\delta}\frac{\varpi_{\lambda}(\delta,q)}{\delta d'} \sum_{n\ge1}\lambda(n)\mathring{V}_{\pm}\left(\frac {n}{\delta d'^2}\r)e \B(\pm\frac{\ol{\delta'b}n}{d'}\B),
\end{align}
where $\delta'=\delta/{(\delta,d)}$ and $d'=d/{(\delta,d)}$.
The Hankel-type transforms $\mathring{V}_\pm:(0,\infty)\rightarrow\mathbb{C}$ are defined by
\[
\mathring{V}_\pm(y)=\int_0^{\infty}V(x)\mJ_\pm\left(4\pi\ssqrt{xy}\r)\d x
\]
with
\begin{align}\label{eqJ1}
\mJ_{+}(x)=2\pi i^{k}J_{k-1}(x), \ \ \ \ \mJ_{-}(x)=0
\end{align}
when $f$ is a holomorphic form of weight $k$, and
\begin{align}\label{eqJ2}
\mJ_+(x)=\frac{\pi i}{\sinh(\pi t)}(J_{2i\kappa}(x)-J_{-2i\kappa}(x)),\quad \mJ_-(x)=4\cosh(\pi \kappa)K_{2i\kappa}(x)
\end{align}
when $f$ is a non-holomorphic form with spectral parameter $\kappa$.
\end{lemma}
\begin{proof}
By Lemma \ref{lemtl}, we write
\begin{align*}
\sum_{ (n,q)=1}\lambda(n)V(n)e\left(\frac{bn}{d}\r)=\sum_{\delta}\varpi_{\lambda}(\delta,q)\sum_{n\ge1}\lambda(n)V(\delta n)e\left(\frac{\delta bn}{d}\r).
\end{align*}
Canceling the common factor $(ab,d)$ in the exponential and applying the Voronoi summation formula (refer to \cite[Lemma 2.3]{BFK+17a}) to the $n$-sum yields the result.
\end{proof}

\begin{remark}\label{remark2}
Since the Selberg eigenvalue conjecture has been proved for the whole modulus group, the spectral parameter $\kappa$ in \eqref{eqJ2} is real when $f$ is a cuspidal Hecke eigenforms of level $1$.
\end{remark}

\subsection{Kuznetsov formula and spectral large sieve}\label{secKlo}
If $Q=uv$ with $(u,v)=1$, then $\mb=1/u$ is a cusp of the Hecke congruence group $\Gamma_0(Q)$,
and its associated Kloosterman sum $S_{\infty,\mb}(m,n;\gammaup)$, with scaling matrix $\sigma_{\mb}=\begin{pmatrix}
     \ssqrt{v} & 0 \\
     u\ssqrt{v} & \frac1{\ssqrt{v}}
\end{pmatrix}$, is defined  precisely when $\gammaup=u\ssqrt{v}w$ for some integer $w$ coprime to $v$. In particularly, it holds
\begin{align}\label{eqSS}
S_{\infty,\mb}(m,n;\gammaup)=e\left(n\frac{\ol{u}}{v}\right)S(m\ol{v},n;uw),
\end{align}
where $\ol{u}$ satisfies $u\ol{u}\equiv1 \pmod{v}$ and $\ol{v}$ satisfies $v\ol{v}=1 \pmod {uw}$
(see \cite[formula (1.6)]{DI82}).

We state the Kuznetsov trace formula from \cite[Theorems 9.4, 9.5, 9.7]{Iwa95}:
\begin{lemma}[Kuznetsov formula]\label{lemKf}
Let $\phi(x)$ be a function on $[0,\infty)$, satisfying $\phi(0)=0$ and $\phi^{(j)}(x)\ll(1+x)^{-2-\ve}$ for $j=0,1,2$. Let $Q=uv$ with $(u,v)=1$, and denote by $\mb=1/u$, a cusp of $\Gamma_0(Q)$. For any positive integers $m, n$,  we have
\begin{align}
\sum_\gamma\frac1{\gammaup}S_{\infty,\mb}(m,n;\gammaup)\phi\left(\frac{4\pi\ssqrt{mn}}{\gammaup}\right)=&\sum_{2\le k\equiv0(\bmod 2)}\sum_{f\in \mathcal{B}_k(Q)}\Gamma(k)\tilde{\phi}(k)\ssqrt{mn}~\ol{\rho}_f(m)~\rho_f(\mb,n)\notag\\
&+\sum_{f\in\mathcal{B}(Q)}\hat{\phi}(\kappa_f)\frac{\ssqrt{mn}}{\cosh(\pi\kappa_f)}~\ol{\rho}_f(m)~\rho_f(\mb,n)\notag\\
&+\frac1{4\pi}\sum_{\mc}\int_{-\infty}^{\infty}\hat{\phi}(\kappa)\frac{\ssqrt{mn}}{\cosh(\pi\kappa)}~\ol{\rho}_{\mc}(m,\kappa)~ \rho_{\mb,\mc}(n,\kappa) \d\kappa\notag
\end{align}
and
\begin{align}
\sum_\gamma\frac1{\gammaup}S_{\infty,\mb}(m,-n;\gammaup)\phi\left(\frac{4\pi\ssqrt{mn}}{\gammaup}\right)= &\sum_{f\in\mathcal{B}(Q)}\breve{\phi}(\kappa_f)\frac{\ssqrt{mn}}{\cosh(\pi\kappa_f)}~\ol{\rho}_f(m)~\rho_f(\mb,-n)\notag\\
&+\frac1{4\pi}\sum_{\mc}\int_{-\infty}^{\infty}\breve{\phi}(\kappa)\frac{\ssqrt{mn}}{\cosh(\pi\kappa)}~\ol{\rho}_{\mc}(m,\kappa) ~\rho_{\mb,\mc}(-n,\kappa) \d\kappa,\notag
\end{align}
where  the Bessel transforms are defined by
\begin{align}
&\tilde{\phi}(k)=4i^k\int_0^\infty\phi(x)J_{k-1}(x)\frac{\d x}x,\notag\\
&\hat{\phi}(\kappa)=2\pi i\int_0^\infty\phi(x)\frac{J_{2i\kappa}(x)-J_{-2i\kappa}(x)}{\sinh(\pi \kappa)}\frac{\d x}x,\notag\\
&\breve{\phi}(\kappa)=8\int_0^\infty\phi(x)\cosh(\pi\kappa)K_{2i\kappa}(x)\frac{\d x}x.\notag
\end{align}
\end{lemma}

The Kuznetsov formula is often used together with the spectral large sieve inequalities; see Deshouillers and Iwaniec \cite[Theorem 2]{DI82}.
\begin{lemma}[Spectral large sieve]\label{lemsls}
Let $\mb$ be a cusp of $\Gamma_0(Q)$ equivalent to $\frac u w$ with $(u,w)=1$ and $w\mid Q$. For $K\ge 1$, $N\ge1$, and complex coefficients $(a_n)_{n\in[N,2N]}$, all three quantities
\[
\sum_{\substack{2\le k\le K\\ k~ \text{even}}}\Gamma(k)\sum_{f\in\mathcal{B}_k(Q)}\B|\sum_{n}a(n)\ssqrt{n}~\rho_f(\mb,n)\B|^2,  \ \ \ \ \ \ \ \ \ \ \ \ \sum_{\substack{f\in\mathcal{B}(Q)\\ |\kappa_f|\le K}}\frac1{\cosh(\pi\kappa_f)}\B|\sum_n a_n\ssqrt{n}~\rho_f(\mb,\pm n)\B|^2,
\]
\[
\sum_{\mc}\int_{-K}^K\frac1{\cosh(\pi\kappa)}\B|\sum_n a_n\ssqrt{n}~\rho_{\mb,\mc}(\pm n,\kappa)\B|^2d\kappa
\]
are bounded by
\[
\ll(K^2+\muup(\mb)N^{1+\ve})\|a_n\|_2^2,
\]
where $\muup(\mb)=Q^{-1}$ for $\mb=\infty$ and $\muup(\mb)=\left(w,\frac Qw\right)Q^{-1}$ otherwise.
\end{lemma}

%Since the Ramanujan-Petersson conjecture is still open for the Maa\ss\  case, the following celebrated result is used in place of the spectral large sieve to avoid the Ramanujan conjecture.
%\begin{lemma}\cite[Theorem 1.3]{BM15}\label{lemRBM}
%Let $s\in\mathbb{N}$, $K$, $N\ge1$, and let $(a_n)_{n\in[N,2N]}$ be a sequence of complex numbers, satisfying $a_n\ll n^\ve$. Then
%\begin{align}
%\sum_{\substack{f\in\mathcal{B}(Q)\\ |\kappa_f|\le K}}\frac1{\cosh(\pi\kappa_f)}\B|\sum_{(n,sQ)=1}a_n\ssqrt{sn}~\rho_f(sn)\B|^2\ll (sQKN)^\ve(s,Q)\B(K+\frac{s^{\frac12}}{Q^{\frac12}}\B)\B(K+\frac{N}{Q^{\frac12}}\B)N.
%\end{align}
%\end{lemma}

\subsection{Bessel function and the Hankel-type transform}
For a given complex number $v$, the Bessel functions satisfy the following well-known bounds (see \cite[Appendix B.4]{Iwa02}):
\begin{align*}
J_v(x)\ll_v\begin{cases}
x^{-\frac12},  &x\ge1,\\
x^{\re (v)}, & x<1,
\end{cases}
\ \ \ \ \ \ \ \ \
Y_v(x), K_v(x)\ll_v \begin{cases}
1,  &x\ge1,\\
(1+|\log|x||)x^{-|\re (v)|}, & x<1.
\end{cases}
\end{align*}
For convenience in presentation, we uniformly express these bounds as
\begin{equation}\label{eqJKY}
J_v, Y_v, K_v\ll_v 1+x^{-\vartheta-\ve}
\end{equation}
with the parameter $\vartheta$ taking the values $-\re (v)$ for $J_v$ and $|\re (v)|$ for $Y_v$, $K_v$. In particular, $K_v$ exhibits rapid decay:
\begin{equation}\label{eqdecayK}
K_v(x)\ll_v x^{-\frac12}e^{-x}\quad \text{for} \quad x\ge 1+|v|^2.
\end{equation}

%Additionally, we require uniform bounds (see \cite[Equations (2.9) \& (2.10)]{BFK+17a}):
%\begin{align}\label{eqJit}
%J_{i\kappa}(x)\ll e^{\frac{|\kappa|}{2}}(|\kappa|+x)^{-\frac12},\ \ \text{for}\ \ \ t\in\mathbb{R},\ x>0,
%\end{align}
%\begin{align}\label{eqJk}
%J_k(x)\ll\min\left(k^{-\frac13}, |x^2-k^2|^{-\frac14}\r),\ k>0, x>0.
%\end{align}

Let $F(x)$ be a smooth function with compact support in $[X,X+X_1]$, where $0<X_1\ll X$, such that
\begin{align}\label{eqf2}
F^{(j)}(x)\ll_{j} X_2^{-j}
\end{align}
for $X_2>0$.
We consider its Hankel-type transform defined by
\begin{align*}
\mathcal{F}_v(y)=\int_{0}^\infty F(x) \mI_{v}\left(4\pi\ssqrt{xy}\right)\d x,
\end{align*}
where $\mI_v$ denotes any of the Bessel functions $(J_v, Y_v$, or $K_v)$.

\begin{lemma}\label{lemwtW1}
Let $F$ be a smooth function as defined above, and let $\mI_v$ denote a Bessel function satisfying \eqref{eqJKY}. For given $v\in\mathbb{C}$, the Hankel-type transform satisfies
\begin{align}\label{eqdwtW}
y^j\mathcal{F}^{(j)}_v(y)\ll_{v,i,j} X_1\left(1+X y\right)^{\frac j2}(1+(Xy)^{-\vartheta-\ve})\left(\left(1+Xy\right)^{-\frac i2}+\left(1+X_2^2 X^{-1}y\right)^{-\frac i2}\r)
\end{align}
for any $i,j\ge 0$ and $y>0$.
In particular, the function $\mathcal{F}_v(y)$ decays rapidly when $y\gg \frac{X^\ve}{\min\{X,X_2^2/X\}}$.
\end{lemma}
\begin{proof}
The trivial bound
\begin{align*}
\mathcal{F}_v(y)\ll  X_1(1+(Xy)^{-\vartheta-\ve})
\end{align*}
follows directly from \eqref{eqJKY}.
Applying the Bessel recursion formula \cite[8.472.1]{GR65}
\begin{align*}
x\mI_{v}(x)=(v+1)\mI_{v+1}(x)+x\mI'_{v+1}(x)
\end{align*}
and integrating by parts, we obtain
\begin{align}\label{eqWW'}
\int_0^\infty F(x)\mI_{v}\left(4\pi\ssqrt{xy}\right)\d x=\int_0^\infty \B(\frac{v+1}{4\pi\ssqrt{xy}}F(x)-\frac{1}{2\pi}\ssqrt{\frac{x}{y}}F'(x)\B)\mI_{v+1}\left(4\pi\ssqrt{xy}\right)\d x.
\end{align}
Repeated applying this formula $i$ times yields
\begin{align*}
\int_0^\infty F(x) \mI_{v}\left(4\pi\ssqrt{xy}\right)\d x
\ll_{v,i} X_1(1+(Xy)^{-\vartheta-\ve})\left(\left(1+Xy\right)^{-\frac i2}+\left(1+X_2^2 X^{-1}y\right)^{-\frac i2}\r).
\end{align*}
For derivatives, we differentiate $j$ times under the integral sign, followed by $i$ times applications of \eqref{eqWW'}. Then a trivially estimate with $\mI'_v(x)=\frac12\left(\pm\mI_{v-1}(x)-\mI_{v+1}(x)\r)$ and the bounds in \eqref{eqJKY} establishes the estimate \eqref{eqdwtW}.
\end{proof}

To handle the Hankel-type transform arising from the Voronoi summation formula, we require its decay properties and smooth asymptotics, established in the following two lemmas.

\begin{lemma}[Decay properties]\label{lemwtW} Let $W$ be a smooth function compactly supported in $[1/2,3]$ and satisfying \eqref{eqW}. For given parameters $b, q, M, N\ge1$, define
\begin{align}\label{eqV0}
\mathring{V}_{\pm}(y,h)=\int_{0}^\infty W\left(\frac {bx-hq}{N}\r)W\left(\frac {bx}M\r) \mJ_{\pm}\left(4\pi\ssqrt{xy}\right)\d x,
\end{align}
where $\mJ_{\pm}$ are as defined in \eqref{eqJ1}-\eqref{eqJ2}, for a cuspidal Hecke eigenform $f$ of level $1$.
We have the following uniform estimate
\begin{align}\label{eqdwtV}
y^jh^l\frac{\partial^{j+l}}{\partial y^j\partial h^l}\mathring{V}_\pm(y,h)\ll_{i,j} \frac{\min\{M,N\}}b \Bigg(\left(1+\frac Mby\right)^{-\frac i2}+\left(1+\frac{\min\{M,N\}^2}{bM}y\right)^{-\frac i2}\Bigg),
\end{align}
valid for any $i,j,l\ge 0$. In particular, the transform decays rapidly for $y\gg\max\left\{\frac{b}{M},\frac{bM}{N^2}\right\}$.
\end{lemma}
\begin{proof}
Since the spectral parameter $\kappa$ in \eqref{eqJ2} is real (see Remark \ref{remark2}), we may take $\vartheta=0$ in the uniform bound \eqref{eqJKY}.
The conclusion then follows from Lemma \ref{lemwtW1} by setting
$X=M/b$ and $X_1=X_2=\min\{M,N\}/b$.
\end{proof}

Next lemma provides a smooth asymptotic formula for the Hankel-type transform with respect to $J_v$, whose proof one may find in \cite[Lemma 17]{BM15}.
\begin{lemma}[Smooth asymptotics]\label{lemdecomJ}
Let $W(x)$ be a smooth function with compact support in $[1/2,3]$ and satisfying \eqref{eqW}. Let $v\in\mathbb{C}$ be a fixed number with $\re(v)\ge0$. For $y, z>0$, define
\[
W^*(y,z)=\int_0^\infty W(x)J_v\left(\ssqrt{y^2+xz^2}\r)dx.
\]
Fix $C\ge 1$ and $A, \ve>0$, when $y\gg z$ we have
\begin{equation}\label{eqsmooth}
W^*(y,z)=W_{+}(y,z)e^{iy}+W_-(y,z)e^{-iy}+O_A(C^{-A}),
\end{equation}
where $W_{\pm}$ (depending on $v$) satisfy
\begin{align*}
y^iz^j\frac{\partial^i}{\partial y^i}\frac{\partial^j}{\partial z^j}W_{\pm}(y,z)
\begin{cases}
=0,  &y/z^2\le C^{-\ve},\\
\ll_{i,j,v} C^{\ve(i+j)}\min\{y^{-\frac12},1\}, & \text{otherwise}.
\end{cases}
\end{align*}
\end{lemma}
After applying the smooth asymptotic formula \eqref{eqsmooth}, the corresponding Bessel transform arising from the Kuznetsov formula satisfies the following upper bound.
\begin{lemma}\label{lemBe}
For $X\ge1$, let
\[
\phi(x)=W\left(\frac xX\r)\exp(\pm i x),
\]
where $W$ is a smooth function of fixed compact support satisfying \eqref{eqW}. Then for $k\in\mathbb{N}$, $\kappa\in\mathbb{R}\cup(-\frac i2,\frac i2)$, we have
\begin{align}\label{eqphi1}
\tilde{\phi}(k), \ \hat{\phi}(\kappa)\ll X^{-\frac12+\ve},
\end{align}
and also
\begin{align}\label{eqphi2}
\tilde{\phi}(k)\ll |k|^{-A},\ \hat{\phi}(\kappa)\ll |\kappa|^{-A}, \ \text{for}~\ k, \ |\kappa|\ge \left(1+X^{\frac12}\r)^{1+\ve}
\end{align}
for any given $A>0$.
\end{lemma}
\begin{proof}
For $k\in\mathbb{N}$, $\kappa\in\mathbb{R}$, see \cite[Corollary \& Remark 1]{Jut99}. For bounded $\kappa\in (-\frac i2,\frac i2)$, the estimate \eqref{eqphi1} follows from \eqref{eqJKY}.
\end{proof}

\section{Sketch of the proof of theorem \ref{themain}}
We decompose the moment into contributions from even and odd characters separately. For $\sigma\in\{1,-1\}$, we define
\begin{equation}\label{eqMsigma}
\mM_{\sigma}(q;a,b)=\frac1{\varphi^*(q)}\ssum_{\substack{\chi\ppmod q\\ \chi(-1)=\sigma}} L\left(\tfrac12,f\otimes\chi\r)L\left(\tfrac12,\ol\chi\r)^2\chi(a\ol{b}).
\end{equation}
Then, as discussed in Remark \ref{rem1}, we have
\[
M_{f,E}(q;a,b)=\mM_{\ve(f)}(q;a,b)
\]
when $f$ is holomorphic, and
\[
M_{f,E}(q;a,b)=\sum_{\sigma\in\{1,-1\}}\mM_{\sigma}(q;a,b)
\]
when $f$ is a Maa\ss~form with root number $\ve(f)=1$.

Substituting the approximate functional equation \eqref{eqafe} into \eqref{eqMsigma} and applying character orthogonality \eqref{eqorth}, we eliminate the averaging to obtain
\[
2\mM_{\sigma}(q;a,b)=\mmM_{\sigma}(q;a,b)+\mmM_{\sigma}(q;b,a),
\]
where
\begin{align*}
\mmM_{\sigma}(q;a,b)=&
\frac1{\varphi^*(q)}\sum_{d\mid q}\vp(d)\mu\left(\frac qd\r)\sum_{\substack{bm\equiv an \ppmod d\\(mn,q)=1}} \frac{\lambda(m)\tau(n)}{(mn)^{\frac12}}V_{f,\sigma}\left(\frac{mn}{q^2}\r)\\
&\ \ \ \ \ \ +\frac{\sigma}{\varphi^*(q)}\sum_{d\mid q}\vp(d)\mu\left(\frac qd\r)\sum_{\substack{bm\equiv-an \ppmod d\\ (mn,q)=1}}\frac{\lambda(m)\tau(n)}{(mn)^{\frac12}}V_{f,\sigma}\left(\frac{mn}{q^2}\r).
\end{align*}
The main term of $\mmM_{\sigma}(q;a,b)$ arises from the diagonal terms $am=bn$, which
 are explicitly computed as
\begin{align*}
MT_{\sigma}^d(q;a,b)&=\frac1{\varphi^*(q)}\sum_{d\mid q}\vp(d)\mu\left(\frac qd\r)\sum_{\substack{bm= an \\(mn,q)=1}}\frac{\lambda(m)\tau(n)}{(mn)^{\frac12}}V_{f,\sigma}\left(\frac{mn}{q^2}\r)\\
&=\sum_{(n,q)=1}\frac{\lambda(an)\tau(bn)}{(ab)^{\frac12}n}V_{f,\sigma}\left(\frac{abn^2}{q^2}\r)\\
&= \frac1{(ab)^{\frac12}}\sum_{a_1\mid a^\infty}\sum_{b_1\mid b^\infty} \frac{\lambda(aa_1b_1)\tau(ba_1b_1)}{a_1b_1} \sum_{(n,abq)=1}\frac{\lambda(n)\tau(n)}{n}V_{f,\sigma}\left(\frac{ab\left(a_1b_1n\r)^2}{q^2}\r).
\end{align*}
Evaluating the sum over $n$ using the definition of $V$, we obtain
\begin{align*}
\sum_{(n,abq)=1}&\frac{\lambda(n)\tau(n)}{n}V_{f,\sigma}\left(\frac{ab\left(a_1b_1n\r)^2}{q^2}\r)\\
&=\prod_{p\mid abq}\left(1-\frac{\lambda(p)}{p}+\frac1{p^2}\r)\left(1-\frac1{p^2}\r)^{-2}\frac{L(1,f)^2}{\zeta(2)}+O\B(\left(\frac{ab\left(a_1b_1\r)^2}{q^{2+\ve}}\r)^{\frac14}\B),
\end{align*}
by moving the contour of the integral in $V$ to $\re(s)=-\frac14+\ve$, with the main term arising from the residue at $s=0$.
The total contribution to $MT_\sigma^d(q;a,b)$ of the error term is
\[
\ll q^{-\frac12+\ve}(ab)^{-\frac14}\sum_{a_1\mid a^\infty}\sum_{b_1\mid b^\infty}\frac{\lambda(aa_1b_1)\tau(ba_1b_1)}{(a_1b_1)^{\frac12}}\ll q^{-\frac12+\ve}a^{-\frac14}b^{\theta_f-\frac14},
\]
an acceptable error for $\theta_f<1/2$.
Based the multiplicativity property of $\lambda$ and $\tau$, an easy calculation then yields the main term of Theorem \ref{themain}.

All off-diagonal terms contribute to the error term. Through a dyadic partition of the unity, we reduce the problem to bounding bilinear expressions of the type
\begin{align}\label{eqA}
E_{M,N}=\frac1{\vp^*(q)}&\sum_{d\mid q}\vp(d)\mu\left(\frac qd\r)\frac{1}{\ssqrt{MN}}\\
&\times\sum_{\substack{bm\equiv \pm an\ppmod d\\ bm\neq an\\(mn,q)=1}}\lambda(m) \tau(n)W_1\left(\frac mM\right)W_2\left(\frac {n}{N}\right)\notag
\end{align}
for $M, N\ge1$, $MN\le q^{2+\ve}$.
The analysis of $E_{M, N}$ requires different techniques depending on the relative sizes of $M$ and $N$. Additionally, we need the following trivial bound, applicable when $M$ and $N$ differ significantly in size.
\begin{lemma}\label{pro1}
Let $E_{M,N}$ be defined as above. We have the following bounds:
\begin{subequations}\begin{align}
&E_{M,N}\ll q^{-1+\ve}(MN)^{\frac12}+q^\ve(M/N)^{\frac12}, \label{eqtbA}\\
&E_{M,N}\ll M^{\theta_f} \left(q^{-1+\ve}(MN)^{\frac12}+q^\ve(N/M)^{\frac12}\r). \label{eqtbB}
\end{align}
\end{subequations}
\end{lemma}
\begin{proof}
Applying the Cauchy-Schwarz inequality to the sum over $m$ yields
\begin{align*}
E_{M,N}\ll \frac1{q^{1-\ve}}\sum_{d\mid q}\vp(d)\frac{1}{\ssqrt{MN}}\B(\sum_{m\asymp M}|\lambda(m)|^2\B)^{\frac12}\B(\sum_{m\asymp M}\B(\sum_{\substack{an\ll AN\\an\equiv bm \ppmod d \\ an\neq bm}}\tau(n)\B)^2\B)^{\frac12}.
\end{align*}
Handling the first $m$-sum via \eqref{eqal} and estimating other sums trivially shows \eqref{eqtbA}.

The argument for \eqref{eqtbB} follows a similar approach, applying the Cauchy--Schwarz inequality to the sum over $n$ rather than $m$. The factor $M^{\theta_f}$ arises from applying the pointwise bound $\lambda(m)\ll m^{\theta_f+\ve}$.
\end{proof}

For balanced terms, we reduce the problem to analyzing the twisted convolution sum $A_q(a,b,M,N)$, as defined in \eqref{eqAdef}.

\begin{theorem}\label{thmAq}
Let $\lambda(m)$ denote the Hecke eigenvalues of a cuspidal Hecke eigenform of level $1$. For $a, b, q\in\mathbb{N}$ with $M\ge N$, we have
\[
A_q(a,b,M,N)\ll  q^\ve\B(\frac{M}{q^{\frac12}}+\frac{(ab,q)^{\frac14}M^{\frac54}N^{\frac14}}{(ab)^{\frac14}q} +\frac{M^{\frac34}N^{\frac14}}{(ab)^{\frac14}q^{\frac14}} +\frac{(ab,q)^{\frac14}MN^{\frac12}}{(ab)^{\frac12}q^{\frac34}}\B).
\]
The same bound holds with $M$ and $N$ interchanged when $N\ge M$,.
\end{theorem}
The proof of Theorem \ref{thmAq} will be presented in Section \ref{secbt}.
Building upon this result, we establish the following bound for balanced terms.
\begin{theorem}\label{thmAMN}
Let $\lambda(m)$ denote the Hecke eigenvalues of a cuspidal Hecke eigenform of level $1$. For $MN\ll q^{2+\ve}$ with $bM\ge aN$, we have
\begin{align}\label{eqbA}
E_{M,N}\ll  q^\ve\left(\frac{ab}{q}\r)^{\frac14}\B(\frac{bM}{aN}\B)^{\frac14}+q^\ve\left(\frac{ab}{q}\r)^{\frac12}\B(\frac{bM}{aN}\B)^{\frac12}.
\end{align}
The same bound holds when $aN\ge bM$, with $aN$ and $bM$ interchanged.
\end{theorem}
\begin{proof}
We focus on the case $bM\ge aN$, as the alternative case follows identically. Let $q_d$ denote the maximal factor of $q$ such that $(q_d,d)=1$. The condition $(mn,q)=1$ in \eqref{eqA} is equivalent to $(m,q_d)=1$ and $(n,q)=1$, both removable via Lemma \ref{lemtl}.
Then, the inner sum in \eqref{eqA} evolves into
\begin{align*}
\mathcal{IS}&=\sum_{l_1}\sum_{l_2}\varpi_\tau(l_1,q)\varpi_\lambda(l_2,q_d)\sum_{\substack{bl_2m\equiv \pm al_1n\ppmod d\\ bl_2m\neq al_1n}}\lambda(m) \tau(n)W_1\left(\frac {l_2m}M\right)W_2\left(\frac {l_1n}{N}\right)\\
&=\sum_{l_1}\sum_{l_2}\varpi_\tau(l_1,q)\varpi_\lambda(l_2,q_d) A_d(al_1,bl_2,bM,aN).
\end{align*}
Applying Theorem \ref{thmAq} and the bound $\varpi_\tau(l_1,q)\varpi_\lambda(l_2,q_d)\ll q^\ve (l_2,q_d)^{\theta_f}$, we obtain
\begin{align*}
\mathcal{IS}\ll q^\ve\sum_{l_1\mid q^2}\sum_{l_2\mid q_d^2}(l_2,q_d)^{\theta_f}\B(\frac{bM}{d^{\frac12}}+\frac{(l_1,d)^{\frac14}bM^{\frac54}N^{\frac14}}{(l_1l_2)^{\frac14}d} +\frac{b^{\frac12}M^{\frac34}N^{\frac14}}{(l_1l_2)^{\frac14}d^{\frac14}} +\frac{(l_1,d)^{\frac14}b^{\frac12}MN^{\frac12}}{(l_1l_2)^{\frac12}d^{\frac34}}\B),
\end{align*}
since $(abl_2,d)=1$. Substituting this into \eqref{eqA} and directly calculating the sum with $\theta_f\le 1/4$, $dq_d\le q$, and $MN\ll q^{2+\ve}$ shows
\begin{align*}
E_{M,N}&=\frac1{\vp^*(q)}\sum_{d\mid q}\vp(d)\mu\left(\frac qd\r)\frac{1}{\ssqrt{MN}}\mathcal{IS}\\
&\ll  q^\ve\left(\frac{ab}{q}\r)^{\frac14}\B(\frac{bM}{aN}\B)^{\frac14}+q^\ve\left(\frac{ab}{q}\r)^{\frac12}\B(\frac{bM}{aN}\B)^{\frac12},
\end{align*}
completing the proof of the lemma.
\end{proof}

We will complete the the proof of Theorem \ref{themain} in the final section by analyzing the unbalanced terms, where the sizes of $M$ and $N$ differ significantly. The analysis of these unbalanced terms builds upon an estimate for bilinear sum of incomplete Kloosterman sums, which is established especially for general moduli in the work of Kerr, Shparlinski, Wu, and Xi \cite{KSWX23}.

\section{Balanced terms and twisted convolution sums}\label{secbt}
This section is devoted to proving Theorem \ref{thmAq}.
\subsection{Initial treatment}
Recall that
\[
A_q(a,b,M,N)=\sum_{\substack{bm\equiv \pm an \ppmod q\\ bm\neq an}}\lambda(m) \tau(n) W\left(\frac {bm}M\r) W\left(\frac {an}N\r),
\]
where we open the divisor function by writing $n=n_1n_2$ with $n_1\le n_2$. To localize the variables $n_1$ and $n_2$, we apply a dyadic partition of unity to both $n_1$ and $n_2$, ensuring $n_1\asymp N_1$ and $n_2\asymp N_2$, where $N_1N_2\asymp N/a$ and $N_1\ll \ssqrt{N/a}$. Define
\[
\mA_q(a,b,M,N,N_1)=\sum_{\substack{bm\equiv \pm an_1n_2 \ppmod q \\ bm\neq an_1n_2\\ (n_1,b)=1}}\lambda(m) W\left(\frac {bm}M\right) W\left(\frac {an_1n_2}N\right)W\left(\frac {n_1}{N_1}\right).
\]
By canceling out $g=(b,n_1)$ in the congruence, we obtain
\begin{align}\label{eqAA}
A_q(a,b,M,N)=\sum_{g\mid b} \mA_q(a,b/g,M/g,N/g,N_1/g)\ll q^\ve|\mA_q(a,b,M,N,N_1)|,
\end{align}
where the final follows since our bound for $\mA_q$ decreases with $g$.
We focus on the treatment of $\mA_q$ for the case $bm\equiv an_1n_2 \pmod q $, as the other case follows identically.

Rewriting the congruence as $bm-an_1n_2=hq$ and rearranging the summation yields
\begin{equation}\label{eqmAq}
\mA_q=\sum_{h\neq0}\sum_{(n_1,b)=1}W\left(\frac {n_1}{N_1}\r)\sum_{bm\equiv hq \ppmod {an_1}}\lambda(m) W\left(\frac {bm}M\r)  W\left(\frac {bm-hq}{N}\r),
\end{equation}
where the summation over $h$ satisfies
\[
0\neq |h|\ll  \frac{\max\{M,N\}}q.
\]
We remove the congruence condition using primitive additive characters modulo $a_1k$ for $a_1\mid a$ and $k\mid n_1$, obtaining
\begin{align*}
\sum_{bm\equiv hq \ppmod {an_1}}&=\frac1{an_1}\sum_{t \ppmod {an_1}}e\left(\frac{bmt-hqt}{an_1}\right)\\
&=\frac1{an_1}\sum_{a_1a_2= a}\sum_{\substack{kr= n_1\\ (r,a_1)=1}}\mathop{\sum\nolimits^*}_{t\ppmod {a_1k}}e\left(\frac{bmt-hqt}{a_1k}\right),
\end{align*}
where we use the unique factorization $d=a_1k$ for divisors $d$ of $an_1$, with $k=(d,n_1)$ and $a_1\mid a$.
Applying the Voronoi summation formula (Lemma \ref{lemcV}), the innermost sum of $\mA_q$ transforms into
\begin{equation}\label{eqinnersum}
\sum_\pm\sum_{a_1a_2= a}\sum_{\substack{kr=n_1\\ (r, a_1)=1}}\frac1{an_1}\frac{1}{a_1k}\sum_m \lambda(m)S(hq\ol{b},\pm m;a_1k)\mathring{V}_{\pm}\left(\frac{m}{a_1^2k^2},h\right),
\end{equation}
where by Lemma \ref{lemwtW}, the Hankel-type transform
\begin{align}\label{eqV0}
\mathring{V}_{\pm}(y,h)=\int_{0}^\infty W\left(\frac {bx-hq}{N}\r)W\left(\frac {bx}M\r) \mJ_{\pm}\left(4\pi\ssqrt{xy}\right)\d x,
\end{align}
decays rapidly when $y\gg\max\left\{\frac{b}{M},\frac{bM}{N^2}\right\}$. Substituting \eqref{eqinnersum} into \eqref{eqmAq} and applying a dyadic partition of unity to both $h$ and $m$,
we conclude that
\begin{equation}\label{eqmAmS}
|\mA_q|\ll q^\ve\sum_{a_1a_2=a}\sum_{r}|\mS(a_1,r)|,
\end{equation}
where the function $\mS(a_1,r)$ is defined by
\begin{equation}\label{eqdefmS}
\mS(a_1,r)=\frac1{aN_1}\sum_{|h|\sim H}\sum_{m\sim M^*} \lambda(m)\sum_{(k,b)=1}\frac{1}{a_1k}W\left(\frac { rk}{N_1}\r)\mathring{V}_{\pm}\left(\frac{m}{a_1^2k^2},h\right)S(hq\ol{b},\pm m;a_1k),
\end{equation}
and the parameters satisfy
\begin{align}\label{eqHM*}
H\ll \frac{\max\{M,N\}}q,\ \ \ \ M^*\ll q^\ve\frac{a_1^2N_1^2}{r^2}\max\left\{\frac{b}{M},\frac{bM}{N^2}\right\}.
\end{align}

\subsection{The case for small $M^*$}
This section is devote to estimating $\mS(a_1,r)$ for the case where $M^*$ is bounded by
\begin{equation}\label{eqsmallM}
M^*\ll q^\ve\frac{a_1^2N_1^2}{r^2}\frac{b}{M}.
\end{equation}
Our analysis relies on the following estimate for sums involving Kloosterman sums, established in \cite[Theorem 1.4]{GWZ25}.
\begin{lemma}\label{thKls}
Let $s\in \mathbb{N}$, $Q=uv$ with $(u,v)=1$, $M, N\ge 1$, and denote by $C=M+N+Q+L+s$. Suppose that $g(m,n,l)$ is a real-valued function of $\mathscr{C}^6$ class, having compact support in $[M,2M]\times[N,2N]\times[L,2L]$ and satisfying
\[
l^{j_1}m^{j_2}n^{j_3}\frac{\partial^{j_1+j_2+j_3}}{\partial l^{j_1}\partial m^{j_2}\partial n^{j_3}}g(m,n,l)\ll C^\ve \quad \text{for}\ \ 0\le j_1, j_2, j_3\le 2.
\]
For any complex sequences $a_m, b_n$, we have
\begin{align*}
%%\label{eqKls1}
\sum_{(l,v)=1}\sum_{m}a_m&\sum_{n} b_n g(m,n,l)S(s m \ol{v},\pm n; lu)\ll C^{\ve}\frac{1+|\log Y|+Y^{-\vartheta_Q}}{1+Y}uv^{\frac12}L\\
&\times\B(1+Y+\frac{s^{\frac12}}{Q^{\frac12}}\B)^{\frac12} \B(1+Y+\frac{(s,Q)^{\frac12}M}{Q^{\frac12}}\B)^{\frac12}\B(1+Y+\frac{N^{\frac12}}{Q^{\frac12}}\B)M^{\frac12}\|a_m\|_\infty^{\frac12}\|b_n\|_2,\notag
\end{align*}
where $\vartheta_Q$ (currently known $\vartheta_Q\le 1/2$) denotes Selberg's spectral gap parameter and
\[
Y=\frac{(sMN)^{\frac12}}{uv^{\frac12}L}.
\]
\end{lemma}

Applying Lemma \ref{lemwtW} yields
\[
h^ik^jm^l\frac{\partial^{i+j+l}}{\partial h^i\partial k^j\partial m^l}\mathring{V}_{\pm}\left(\frac{m}{a_1^2k^2},h\right)\ll_{i,j,l} \frac{\min\{M, N\}}{b}.
\]
Hence, we can express $\mS(a_r,r)$ as the required form of Lemma \ref{thKls}:
\begin{equation}
\mS(a_1,r)=\frac{\min\{M, N\}}{aa_1b}\frac{r}{N_1^2}\sum_{|h|\sim H}\sum_{m\sim M^*} \lambda(m)\sum_{(k,b)=1}g(h,m,k)S(hq\ol{b},\pm m;a_1k),
\end{equation}
where $g(h,m,k)$ is a smooth function satisfying the condition in Lemma \ref{thKls}. Applying Lemma \ref{thKls} shows
\begin{align*}
\mS(a_1,r)\ll& \frac{\min\{M, N\}}{a_2rq^{\frac12}}\frac{1+|\log Y|+Y^{-\vartheta_{a_1b}}}{1+Y}Y\\
&\B(1+Y+\left(\frac{q}{a_1b}\right)^{\frac12}\B)^{\frac12} \B(1+Y+\frac{(a_1b,q)^{\frac12}H}{a_1^{\frac12}b^{\frac12}}\B)^{\frac12}\B(1+Y+\left(\frac{M^*}{a_1b}\r)^{\frac12}\B)
\end{align*}
with
\[
Y=\frac{r\ssqrt{HqM^*}}{a_1\ssqrt{b}N_1}.
\]
Since the bound grows as $H$ or $M^*$ increase, we take
\[
H=\frac{\max\{M,N\}}q,\ \ \ \ M^*=q^\ve\frac{a_1^2N_1^2}{r^2}\frac{b}{M},
\]
leading to
\[
Y=q^\ve\B(\frac{\max\{M,N\}}{M}\B)^{\frac12}\gg1,
\]
and then
\begin{align*}
\mS(a_1,r)\ll&  q^\ve\frac{\min\{M,N\}}{a_2r\ssqrt{q}}\B(\frac{\max\{M,N\}^{\frac12}}{M^{\frac12}}+\B(\frac{q}{a_1b}\B)^{\frac12}\B)^{\frac12} \\
&\times \B(\frac{\max\{M,N\}^{\frac12}}{M^{\frac12}}+\frac{(a_1b,q)^{\frac12}\max\{M,N\}}{a_1^{\frac12}b^{\frac12}q}\B)^{\frac12}
\B(\frac{\max\{M,N\}^{\frac12}}{M^{\frac12}}+\frac{a^{\frac12}N_1}{M^{\frac12}}\B).
\end{align*}
Substituting this into \eqref{eqmAmS} and applying the facts
\[
\frac{a^{\frac12}N_1}{M^{\frac12}}\ll \frac{\max\{M,N\}^{\frac12}}{M^{\frac12}} ~~( \text{since}~ N_1\le \ssqrt{N/a}),\quad \min\{M,N\}\max\{M,N\}=MN,
\]
 we obtain
\begin{align*}
\mA_q\ll  q^\ve\B(\frac{N}{q^{\frac12}}+\frac{(ab,q)^{\frac14}\left(M^{\frac54}N^{\frac14}+M^{\frac14}N^{\frac54}\r)}{(ab)^{\frac14}q} +\frac{M^{\frac14}N^{\frac34}}{(ab)^{\frac14}q^{\frac14}} +\frac{(ab,q)^{\frac14}M^{\frac12}N}{(ab)^{\frac12}q^{\frac34}}\B).
\end{align*}
Applying this bound to \eqref{eqAA} establishes the claimed estimate in Theorem \ref{thmAq}.

\subsection{The case for large $M^*$}
In this section, we focus on the case where $M^*$ is large, satisfying
\begin{align}\label{eqlargeM}
 q^\ve\frac{a_1^2N_1^2}{r^2}\frac{b}{M}\ll M^*\ll q^{\ve}\frac{a_1^2N_1^2}{r^2}\frac{bM}{N^2},
\end{align}
which occurs only when $M\gg N^{1+\ve}$. A larger $M^*$ also leads to a larger $Y$, and a direct application of Lemma \ref{thKls} fails to provide a necessary estimate for $\mS(a_1,r)$. To overcome this, we leverage the special Bessel transform estimate established in Lemma \ref{lemBe}.

By the support of $W_2\left(\frac {bm-hq}{N}\right)$, the summation over $h$ in $\mS(a_1,r)$ is restricted to the range
\[
H\le h\le 2H \ \ \ \ \text{with}\ \ \ \ H\asymp \frac{M}q.
\]
Based on \eqref{eqdecayK},
the rapid decay property of $\mJ_-\left(x\frac{m}{a_1^2k^2}\r)$ in \eqref{eqV0} yields the bound
\[
\mathring{V}_{-}\left(\frac{m}{a_1^2k^2},h\right)\ll_j (MN)^{-j}
\]
for any integer $j\ge 1$. This follows because
\[
x\frac{m}{a_1^2k^2}\gg\frac{M}{b}\frac{r^2 m}{a_1^2N_1^2}\gg \frac{M}{b}\frac{r^2}{a_1^2N_1^2}M^*\gg q^\ve
\]
by virtue of \eqref{eqlargeM}.
Consequently, it only remains to treat the sum involving $\mathring{V}_{+}$.

Using the identity
\[
S(hq\ol{b}, m;a_1k)=e\left(-m\frac{\ol{a}_1}{b}\right)S_{\infty,1/a_1}(hq,m;\gammaup),
\]
where $\gammaup=a_1\ssqrt{b}k$ and $Q=a_1b$, we derive from \eqref{eqdefmS} that
\begin{align*}
\mS(a_1,r)=&\frac{\sqrt{b}}{aN_1}\sum_{h\sim H}\sum_{m\sim M^*} \lambda(m)e\left(-m\frac{\ol{a}_1}{b}\right)\\
&\times\sum_{\gammaup}\frac 1{\gammaup}S_{\infty,1/a_1}(hq, m;\gammaup)W\left(\frac {r\gammaup}{a_1\ssqrt{b}N_1}\r)\mathring{V}_{+}\left(\frac{bm}{\gammaup^2},h\right).
\end{align*}
Performing the variable substitution $(bx-hq)/N\rightarrow x$ in \eqref{eqV0} yields
\[
\mathring{V}_{+}\left(\frac{bm}{\gammaup^2},h\right)=\frac{N}{b}\Omega\Bigg(\frac{4\pi\ssqrt{hqm}}{\gammaup},\frac{4\pi\ssqrt{mN}}{\gammaup},h,m\Bigg),
\]
where
\[
\Omega(y,z,h,m)=\int_{0}^\infty W(x)W\left(\frac {hq+xN}M\r) \mJ_{+}\left(\ssqrt{y^2+xz^2}\right)\d x
\]
with
\begin{equation}\label{eqyz}
y\asymp Y:=\frac{r\ssqrt{MM^*}}{a_1\ssqrt{b}N_1}\gg  q^\ve,\quad  z\asymp \frac{r\ssqrt{M^*N}}{a_1\ssqrt{b}N_1},\quad h\asymp H,\quad m\asymp M^*.
\end{equation}
For $M\gg q^\ve N$, it is easy to see from \eqref{eqlargeM} and \eqref{eqyz} that
\[
y\gg z\quad \text{and}\quad y\gg z^2q^{-\ve}.
\]
Then Lemma \ref{lemdecomJ} yields the decomposition
\[
\Omega(y,z,h,m)=\Omega_+(y,z,h,m)e^{iy}+\Omega_-(y,z,h,m)e^{-iy}+O_A\left(q^{-A}\r),
\]
where the functions $\Omega_{\pm}$ satisfy
\begin{align}\label{eqWpmd}
y^{j_1}z^{j_2}h^{j_3}m^{j_4}\frac{\partial^{j_1+j_2+j_3+j_4}}{\partial y^{j_1}\partial z^{j_2}\partial h^{j_3}\partial m^{j_4}}\Omega_{\pm}(y,z,h,m)\ll_{j_1,j_2,j_3,j_4}q^{(j_1+j_2+j_3+j_4)\ve} Y^{-\frac12}.
\end{align}
After separating variables in $\Omega_\pm$ via the Mellin transform, we bound $\mS(a_1,r)$ by sums of the form
\begin{equation*}
\mS(a_1,r)\ll q^\ve\frac{N Y^{-\frac12}}{a\ssqrt{b}N_1}\B|\sum_{H<h\le 2H}\alpha_h\sum_{M^*<m\le 2M^*} \beta_m\lambda(m)\sum_{\gammaup}\frac1{\gammaup}S_{\infty,1/a_1}(hq,m;\gammaup)\Theta \B(\frac{4\pi\ssqrt{hqm}}{\gammaup}\B)\B|
\end{equation*}
for some coefficients satisfying $\alpha_h\ll h^\ve$, $\beta_m\ll m^\ve$ and a test function
\[
\Theta (y)=W\B(\frac yY\B)\exp(\pm iy)y^{-u}
\]
with $\re(u)=\ve$.

An application of the Kuznetsov formula (Lemma \ref{lemKf}) followed by spectral truncation at parameter $Y^{\frac12+\ve}$ ( Lemma \ref{lemBe}) produces
\begin{equation*}
\mS(a_1,r)\ll q^\ve\frac{N Y^{-\frac12}}{a\ssqrt{b}N_1}\B|\sum_{H<h\le 2H}\alpha_h\sum_{M^*<m\le 2M^*} \beta_m\lambda(m)\left(\mH+\mM+\mE\r)\B|,
\end{equation*}
where the spectral contributions decompose as
\[
\mH=\sum_{\substack{2\le k\le Y^{\frac12+\ve}\\ k ~\text{even}}}\sum_{f\in \mathcal{B}_k(Q)}\tilde{\Theta} (k)\ssqrt{hqm}~ \bar\rho_f(hq)~\rho_f(a_1^{-1},m),
\]
\[
\mM=\sum_{\substack{f\in \mathcal{B}(Q)\\ \kappa_f\ll Y^{\frac12+\ve}}}\hat{\Theta} (\kappa_f)\frac{\ssqrt{hqm}~ \bar\rho_f(hq)~\rho_f(a_1^{-1},m)}{\cosh(\pi \kappa_f)},
\]
\[
\mE=\frac1{4\pi}\sum_{\mc}\int_{|\kappa|\ll Y^{\frac12+\ve}}\hat{\Theta} (\kappa)\ssqrt{hqm} ~\frac{\bar\rho_{\mc}(hq,\kappa)~\rho_{1/{a_1},\mc}(m,\kappa)}{\cosh(\pi\kappa)}\d \kappa,
\]
representing the holomorphic cusp forms, Maa\ss\ cusp forms, and the Eisenstein series respectively.
We foucus our detailed analysis on the Maa\ss\  case (the holomorphic and Eisenstein cases are identical). Since $\hat{\Theta} (\kappa_f)\ll Y^{-\frac12+\ve}$ (see \eqref{eqphi1}), the contribution of the Maa\ss\  case is bounded by
\[
\ll q^\ve\frac{NY^{-1}}{a\ssqrt{b}N_1}\sum_{\substack{f\in \mathcal{B}(Q)\\ \kappa_f\ll Y^{\frac12+\ve}}}\B|\sum_{H<h\le 2H}\alpha_h\sum_{M^*<m\le 2M^*} \beta_m\lambda(m)\frac{\ssqrt{hqm} ~\bar\rho_f(hq)~\rho_f(a_1^{-1},m)}{\cosh(\pi \kappa_f)}\B|,
\]
which, after applying the Cauchy-Schwarz inequality, can be bounded by
\begin{align*}
\ll q^\ve\frac{NY^{-1}}{a\ssqrt{b}N_1}&\B(\sum_{\substack{f\in \mathcal{B}(Q)\\ \kappa_f\ll Y^{\frac12+\ve}}}\frac{1}{\cosh(\pi \kappa_f)}\B|\sum_{H<h\le 2H}\alpha_h \ssqrt{hq} ~\rho_f(hq)\B|^2\B)^{\frac12}\\
&\times \B(\sum_{\substack{f\in \mathcal{B}(Q)\\ \kappa_f\ll Y^{\frac12+\ve}}}\frac{1}{\cosh(\pi \kappa_f)}\B|\sum_{M^*<m\le 2M^*}\beta_m\lambda(m)\ssqrt{m} ~\rho_f(a_1^{-1},m)\B|^2\B)^{\frac12}.
\end{align*}
We may handle the second spectral sum via Lemma \ref{lemsls}, but for the first spectral sum, an application of Lemma \ref{lemsls} leads to an extra factor of $q$ to the final estimate. To address this issue, we apply the following estimate, established in \cite[Theorems 11.3-11.6]{GWZ25}.
\begin{lemma}\label{thmsls}
Let $Q,s\in\mathbb{N}$, $K, N\ge 1$, and let $(a_n)_{n\in[N,2N]}$ be complex coefficients. Then we have
\[
\sum_{\substack{2\le k\le K\\ k~ \text{even}}}\Gamma(k)\sum_{f\in\mathcal{B}_k(Q)}\B|\sum_{n}a_n\ssqrt{sn}~\rho_f(sn)\B|^2,  \ \ \ \sum_{\mc}\int_{-K}^K\frac1{\cosh(\pi\kappa)}\B|\sum_n a_n\ssqrt{sn}~\rho_{\mc}(\pm sn,\kappa)\B|^2\d\kappa,
\]
\[
\sum_{\substack{f\in\mathcal{B}(Q)\\ |\kappa_f|\le K}}\frac1{\cosh(\pi\kappa_f)}\B|\sum_{n}a_n\ssqrt{sn}~\rho_f(sn)\B|^2
\]
are bounded by
\[
\ll (sQKN)^\ve\B(K+\frac{s^{\frac12}}{Q^{\frac12}}\B)\B(K+\frac{(s,Q)^{\frac12}N}{Q^{\frac12}}\B)N\|a_n\|_\infty.
\]
\end{lemma}
Applying Lemma \ref{thmsls} to the first spectral sum and Lemma \ref{lemsls} to the second one reduces the Maa\ss\  contribution to
\begin{equation*}
\ll q^\ve\frac{NY^{-1}}{a\ssqrt{b}N_1}\B(H^{\frac12}{M^*}^{\frac12}\left(Y+\frac{q}{a_1b}\r)^{\frac14} \left(Y+\frac{(a_1b,q)H^2}{a_1b}\r)^{\frac14}\left(Y+\frac{M^*}{a_1b}\r)^{\frac12}\B).
\end{equation*}
Taking
\[
 H=\frac{M}q,\ \ \ M^*=  q^\ve\frac{a_1^2N_1^2}{r^2}\frac{bM}{N^2}, \ \ \ Y=\frac{r\ssqrt{MM^*}}{a_1\ssqrt{b}N_1}=\frac{M}{N},
\]
and observing that
$\frac{M^*}{a_1b}\ll Y$ ( since $N_1^2\ll N/a$), we obtain
for the Maa\ss\  contribution
\begin{align*}
&\ll q^\ve\frac{N}{r\ssqrt{q}}\left(Y+\frac{q}{ab}\r)^{\frac14} \left(Y+\frac{(ab,q)H^2}{ab}\r)^{\frac14}Y^{\frac12}\\
&\ll  q^\ve r^{-1}\B(\frac{M}{q^{\frac12}}+\frac{(ab,q)^{\frac14}M^{\frac 54}N^{\frac14}}{(ab)^{\frac14}q} +\frac{M^{\frac 34}N^{\frac14}}{(ab)^{\frac14}q^{\frac14}}+\frac{(ab,q)^{\frac14}MN^{\frac12}}{(ab)^{\frac12}q^{\frac 34}}\B).
\end{align*}

The same bounds are established by identical arguments for the holomorphic and Eisenstein cases.
We therefore conclude with the final estimate
\[
\mS(a_1,r)\ll  q^\ve r^{-1}\B(\frac{M}{q^{\frac12}}+\frac{(ab,q)^{\frac14}M^{\frac 54}N^{\frac14}}{(ab)^{\frac14}q}+\frac{M^{\frac 34}N^{\frac14}}{(ab)^{\frac14}q^{\frac14}} +\frac{(ab,q)^{\frac14}MN^{\frac12}}{(ab)^{\frac12}q^{\frac 34}}\B),
\]
which, substituted into \eqref{eqmAmS} and then into \eqref{eqAA}, establishes the required bound in Theorem \ref{thmAq}.

\section{Unbalanced terms}
In this section, we complete the proof of Theorem \ref{themain} by bounding the unbalanced terms. Since our arguments proceed identically for both cases $bm\equiv an\pmod d$ and $bm\equiv -an\pmod d$ in \eqref{eqA}, we focus on the former one, which is given by
\begin{align}\label{eqEMN1}
E_{M,N}=\frac1{\vp^*(q)}&\sum_{d\mid q}\vp(d)\mu\left(\frac qd\r)\frac{1}{\ssqrt{MN}}\\
&\times\sum_{\substack{bm\equiv an\ppmod d\\ bm\neq an\\(mn,q)=1}}\lambda(m) \tau(n)W\left(\frac mM\right)W\left(\frac {n}{N}\right).\notag
\end{align}
Here, $bM$ and $aN$ are of significantly different sizes. This is analogous to $\mB_{h,k}(M,N)$ in \cite[(9.2)]{GWZ25}, but the presence of Hecke eigenvalues $\lambda(m)$ breaks the symmetry between $m$ and $n$ in the current setting.

When $bM\ll aN$, the Hecke eigenvalues $\lambda(m)$ introduce no essential difference due to the average validity of the Ramanujan--Petersson conjecture. Consequently, an analogous argument to that in \cite[Section 10.1]{GWZ25} yields
\begin{equation}\label{eqEMNub1}
E_{M,N}\ll q^{-\frac1{20}+\ve}a^{\frac3{10}}.
\end{equation}

We now consider the case where $bM\gg aN$. If $q$ is a prime modulus, the most efficient approach is to apply the Voronoi summation formula to both variables $m$ and $n$, effectively swapping $M$ and $N$. This reduces the problem to the $bM\ll aN$ case, albeit with a slightly longer sum (cf. \cite[Section 6.4.2]{BFK+17a}). For composite moduli, however, the coprime condition $(mn,q)=1$ introduces complications in applying the Voronoi formula ( Lemma \ref{lemcV}), potentially leading to very long exponential sums even when $M$ and $N$ are large. Consequently, we derive sharp estimates directly for this case, leveraging Weil's bound and the following specialized estimate for bilinear forms of incomplete Kloosterman sums.
\begin{lemma}[{\cite[Theorem 2.4]{KSWX23}}]\label{lemDS}
Let $q$ be a positive integer and $\alpha_a, \beta_b$ be sequences of complex numbers. For any $A,B\ge1$, the estimate
\begin{equation*}
%%\label{eqDS1}
\sum_{a\le A}\alpha_a\B|\sum_{\substack{b\le B\\ (b,q)=1}} \beta_b e\B(\frac{ca\ol{b}}{q}\B) \B|\ll\vert\bm\alpha\vert_2|\vert \bm\beta\vert_\infty A^{\tfrac12}Bq^\ve\left(A^{-\frac{1}2}B^{-\frac{1}4}q^{\frac{1}4}+A^{-\frac12}+q^{-\frac12}+B^{-\frac{1}2}\r)
\end{equation*}
holds uniformly in $c$ with $(c,q)=1$.
\end{lemma}

We detect the congruence condition in \eqref{eqEMN1} using primitive additive characters modulo $c\mid d$ and apply the Voronoi formula (Lemma \ref{lemcV}) to the $m$-sum. This yields
\begin{align*}
E_{M,N}=\frac1{\vp^*(q)}&\sum_{c\mid d\mid q}\frac{\vp(d)}{d}\mu\left(\frac qd\r)\frac{1}{\ssqrt{MN}}
\sum_{\delta}\varpi_{\lambda}(\delta,q)\frac{M}{\delta c'}\\
&\times\sum_{(n,q)=1} \sum_{m}\lambda(m)
\mathring{W}_\pm\left(\frac {mM}{\delta c'^2}\right) W\left(\frac {n}{N}\right)S(\ol{a\delta'}m,\pm bn; c')\notag
\end{align*}
with $\delta'=\delta/(\delta,c)$ and $c'=c/(\delta,c)$. Bounding $\varpi_{\lambda}(\delta,q)$ via the estimate
\begin{equation*}
\varpi_{\lambda}(\delta,q)\ll (\delta,q)^{\theta_f}\le ((\delta,c)q_c)^{\theta_f}= \left(cq_c/c'\r)^{\theta_f}\le (q/c')^{\theta_f}
\end{equation*}
and applying Weil's bound to the Kloosterman sum, we evaluate all the sums trivially, obtaining
\begin{equation}\label{eqEMNWeil}
E_{M,N}\ll\frac{q^\ve}{\vp^*(q)}\sum_{c'\mid q}\left(\frac{q}{c'}\r)^{\theta_f}\frac{c'^{\frac32}N^{\frac12}}{M^{\frac12}}
\ll q^{\ve}\left(\frac {qN}M\r)^{\frac12},
\end{equation}
valid for $\theta_f<1/2$.

We now expand the divisor function $\tau(n)$ by writing $n=n_1n_2$, where $n_1\asymp N_1, n_2\asymp N_2$, $N_1N_2\asymp N$, and $N_1\le N_2$.
After removing $(n_2,q_d)=1$ using the M\"{o}bius function, we apply the Poisson summation formula to the $n_2$-sum. An identical argument to that in \cite[Section 9]{GWZ25} then efficiently reduces $E_{M,N}$ to the following structured form
\begin{align}\label{eqEpm}
E_{M,N}\ll& \frac{q^\ve}{\vp^*(q)}\sum_{d\mid q}\sum_{g\mid q_d}\frac{\vp(d)}{dg}\frac{N_2}{\sqrt{MN}}\\
&\times\B|\sum_{h\neq0}\sum_{(m,q)=1}\sum_{(n_1,q)=1} \lambda(m) e\left(\frac{bhm\ol{agn_1}}{d}\r) W\left(\frac {m}M\right) W\left(\frac {n_1}{N_1}\right)\wh{W}\left(\frac {hN_2}{gd}\right)\B|,\notag
\end{align}
which is an analogy of \cite[(9.3)]{GWZ25} but with $\tau(m)$ replaced by $\lambda(m)$.

We cannot directly apply Lemma \ref{lemDS} to \eqref{eqEpm} because the $m$-sum is long. To shorten the sum, we apply the Voronoi summation formula (Lemma \ref{lemcV}) to the variable $m$, which yields
\begin{align*}
E_{M,N}\ll&\frac{q^\ve}{\vp^*(q)}\sum_{d\mid q}\frac{N_2}{\sqrt{MN}}\sum_{\delta}|\varpi_{\lambda}(\delta,q)|\frac{M}{\delta d'}\\
&\times\B|\sum_{(h,q)=1}\sum_{(n_1,q)=1} \sum_{m}\lambda(m)e\B(\pm\frac{amn_1\ol{\delta'bh}}{d'}\B)
\mathring{W}_\pm\left(\frac {m}{M^*}\right) W\left(\frac {n_1}{N_1}\right)\wh{W}\left(\frac {h}{H}\right)\B|,\notag
\end{align*}
where
\[
\delta'=\delta/(\delta,d),\quad d'=d/(\delta,d),\quad 1\ll H\ll d/N_2,\quad 1\ll M^*\ll \delta d'^2/M.
 \]
Here, we have removed the $g$-sum and added the condition $(h,q)=1$ to the $h$-sum at no additional cost.
Setting $l=mn_1$ and $k=h$, we apply Lemma \ref{lemDS} with
 \[
 L=\frac{\delta d'^{2+\ve}N_1}{M}\gg \frac{dd'N_1}{M}\quad \text{and}\quad K=H,
 \]
obtaining
\begin{align}\label{eqEMN+W}
E_{M,N}\ll& \frac{q^\ve}{\vp^*(q)}\sum_{d'\mid d\mid q}\left(\frac{q}{d'}\r)^{\theta_f}\frac{d'HN^{\frac12}}{M^{\frac12}}
\B(\frac{M^{\frac12}}{d'^{\frac34}H^{\frac14}N_1^{\frac12}}+\frac{M^{\frac12}}{(dd'N_1)^{\frac12}}+\frac1{d'^{\frac12}}+\frac1{H^{\frac12}}\B)\notag\\
&\ll q^\ve N_2^{-\frac14}+q^{\ve}N_2^{-\frac12}\left(\frac {qN}M\r)^{\frac12},
\end{align}
valid for $\theta_f<1/4$.

Let $\eta=\frac{1-2\theta_f}{12+12\theta_f}$ (to be specified later) and set
\[
a=q^\alpha,\quad b=q^\beta,\quad M=q^\mu,\quad  N=q^\nu, \quad N_2=q^{\nu_2}.
\]
By \eqref{eqtbB}, \eqref{eqbA}, and \eqref{eqEMNWeil}, it suffices to prove $E_{M,N}\ll q^{-\eta+\ve}$
for the parameters satisfying
\begin{align}\label{eqrange}
2-2\eta\le (1+2\theta_f)\mu+\nu,&\quad 1-2\beta-4\eta\le \mu-\nu\le 1+2\eta,\quad \nu_2\ge\tfrac12 \nu.
\end{align}
The estimate \eqref{eqEMN+W} then yields
\[
E_{M,N}\ll q^{-\frac14\nu_2+\ve}+q^{-\frac12(\nu_2+\mu-\nu-1)+\ve},
\]
while
\begin{align*}
-\tfrac14\nu_2&\le -\frac{((1+2\theta_f)u+v)-(1+2\theta_f)(u-v)}{16+16\theta_f}\\
&\le -\frac{(2-2\eta)-(1+2\theta_f)(1+2\eta)}{16+16\theta_f}\le -\eta
\end{align*}
for $\eta\le\frac{1-2\theta_f}{12+12\theta_f}$, and
\begin{align*}
-\tfrac12(\nu_2+\mu-\nu-1)&\le -\frac{((1+2\theta_f)u+v)-(1+2\theta_f)(u-v)}{8+8\theta_f} -\tfrac12(\mu-\nu-1)\\
&\le -\frac{2-2\eta-(1+2\theta_f)(1-2\beta-4\eta)}{8+8\theta_f}+\beta+2\eta\le -\eta
\end{align*}
for $0<\eta\le \frac{1-2\theta_f-(6+4\theta_f)\beta }{22+16\theta_f}$, subject to the condition given in \eqref{eqrange}. This yields
\[
E_{M,N}\ll q^{-\frac{1-2\theta_f}{22+16\theta_f}+\ve}b^{\frac{3+2\theta_f }{11+8\theta_f}},
\]
which, combining with \eqref{eqEMNub1}, completes the proof of Theorem \ref{themain}.

\end{document}